\documentclass[preprint,12pt,3p]{elsarticle}
\usepackage{amsmath,amsthm,amsfonts,amssymb}
\usepackage{hyperref}
\usepackage{yhmath}
\usepackage{MnSymbol}  
\usepackage[normalem]{ulem}
\usepackage{sectsty}

\usepackage{colortbl}
\usepackage[dvipsnames]{xcolor}
\newtheorem{theorem}{Theorem}[section]
\newtheorem{lemma}[theorem]{Lemma}

\newtheorem{remark}[theorem]{Remark}
\newtheorem{thm}{Theorem}[section]
\newtheorem{cor}{Corollary}[section]

\numberwithin{equation}{section}

\definecolor{darkslategray}{rgb}{0.18, 0.31, 0.31}
\definecolor{warmblack}{rgb}{0.0, 0.26, 0.26}
\definecolor{astral}{RGB}{46,116,181}
\subsectionfont{\color{astral}}
\sectionfont{\color{astral}}
\journal{....................... }

\begin{document}
	\begin{frontmatter}
		\title{ \textcolor{warmblack}{\bf On $A$-numerical radius inequalities for $2\times2$ operator matrices }}

	\author[label1]{Nirmal Chandra Rout}\ead{nrout89@gmail.com}
	
	\author[label2]{Satyajit Sahoo\corref{cor2}}\ead{satyajitsahoo2010@gmail.com}

		\author[label1]{Debasisha Mishra}\ead{dmishra@nitrr.ac.in}
		
		\address[label1]{Department of Mathematics, National Institute of Technology Raipur, Raipur-492010, India}
		\address[label2]{P.G. Department of Mathematics, Utkal University, Vanivihar, Bhubaneswar-751004, India}

		\cortext[cor2]{Corresponding author}

		\begin{abstract}
			\textcolor{warmblack}{
 				 Let ($\mathcal{H}, \langle . , .\rangle )$ be a complex Hilbert space and $A$ be a positive bounded linear operator on it. Let $w_A(T)$  be the $A$-numerical radius and $\|T\|_A$ be the $A$-operator seminorm of an operator $T$ acting on the semi-Hilbertian space $(\mathcal{H}, \langle .,.\rangle_A),$ where $\langle x, y\rangle_A:=\langle Ax, y\rangle$ for all $x,y\in \mathcal{H}$. In this article, we establish  several upper and lower bounds for $B$-numerical radius of $2\times 2$ operator matrices,  where $B=\begin{bmatrix}
    A & 0\\
     0 & A
    \end{bmatrix}$. 
    Further, we prove some refinements of earlier $A$-numerical radius inequalities for  operators.}
		\end{abstract}
		
		\begin{keyword}
			$A$-numerical radius; Positive operator; Semi-inner product; Inequality; Operator matrix
		\end{keyword}
 	\end{frontmatter}
	
	\section{Introduction}\label{intro}
	Let $\mathcal{H}$ be a complex Hilbert space with inner product $\langle \cdot,\cdot\rangle$ and $\mathcal{B}(\mathcal{H})$ be the $\mathbb{C}^*$- algebra of all bounded linear operators on $\mathcal{H}$.
 For $T\in
\mathcal{L(H)}$, the {\it numerical range} of $T$ is defined as
 $$W(T)=\{\langle Tx, x \rangle: x\in \mathcal{H}, \|x\|=1
\}.$$ The {\it numerical radius} of $T$, denoted by $w(T)$,  is
defined as $ w(T)=\displaystyle\sup\{|z|: z\in W(T) \}.$
It is well-known that
$w(\cdot)$ defines a norm on $\mathcal{H}$, and is equivalent to the
usual operator norm $\|T\|=\displaystyle \sup  \{ \|Tx \|: x\in \mathcal{H}, \|x\|=1 \}.$ In fact, for
every $T \in \mathcal{L(H)}$, 
\begin{align}\label{p3100}
\frac{1}{2}\|T\|\leq w(T)\leq \|T\|.
\end{align}
 One may refer \cite{MBKS,TY,SND,SND1,SND2,HirKit} for several generalizations, refinements and applications of numerical radius inequalities in different settings which appeared in the last decade.
 Let $\|\cdot\|$ be the norm induced from $\langle \cdot,\cdot\rangle.$
 A selfadjoint operator $A\in\mathcal{B}(\mathcal{H})$ is called {\it positive} if $\langle Ax, x\rangle \geq 0$ for all $x \in \mathcal{H}$, and is called {\it strictly positive} if  $\langle Ax, x\rangle > 0$ for all non-zero $x\in \mathcal{H}$. We denote a positive (strictly positive) operator $A$ by $A \geq 0$ ($ A > 0$). Let $B$ be a $2\times 2$ diagonal operator matrix, in which each of the diagonal entries is a positive operator $A$. Through out this article, $A$ is always assumed to a positive operator. Clearly, if $A$ is a positive operator, it induces a positive semidefinite sesquilinear form, $\langle \cdot,\cdot\rangle_A: \mathcal{H}\times\mathcal{H}\rightarrow\mathbb{C}$ defined by $\langle x, y \rangle_A=\langle Ax,y\rangle,$ $x,y\in\mathcal{H}.$ 
 Let $\|\cdot\|_A$ denote the semi-norm on $\mathcal{H}$ induced by $\langle \cdot, \cdot \rangle_A,$ i.e., $\|x\|_A=\sqrt{\langle x, x \rangle_A}$ for all $x \in \mathcal{H}.$ It is easy to verify that $\|x\|_A$ is a norm
if and only if $A$ is a strictly positive operator. Also, $(\mathcal{H}, \|\cdot\|_A)$ is complete if and only if the range of $A$ ($\mathcal{R}(A)$) is closed in $\mathcal{H}.$ For $T \in\mathcal{B}\mathcal{(H)}$, $A$-operator seminorm of $T$, denoted as
$\|T\|_A,$ is defined as
$$\|T\|_A:=\sup_{x\in \overline{\mathcal{R}(A)},~x\neq 0}\frac{\|Tx\|_A}{\|x\|_A}=\inf\left\{c>0: \|Tx\|_A\leq c\|x\|_A,x\in \overline{\mathcal{R}(A)}\right\}<\infty.$$
We set $\mathcal{B}^A\mathcal{(H)}:=\{T\in \mathcal{B(H)}:\|T\|_A<\infty\}.$ It can be seen that $\mathcal{B}^A\mathcal{(H)}$ is not generally a subalgebra of $\mathcal{B(H)}$, and $\|T\|_A=0$ if and only if $ATA=0.$ For $T\in\mathcal{B}^A\mathcal{(H)},$ we also have 
$$\|T\|_A=\sup \{|\langle Tx,y\rangle_A|: x,y\in \overline{\mathcal{R}(A),} ~\|x\|_A=\|y\|_A=1\}.$$
If $AT\geq 0$, then the operator $T$ is called {\it $A$-positive}.  Note that if $T$ is $A$-positive, then 
$$\|T\|_A=\sup \{\langle Tx,x\rangle_A: x\in \mathcal{H},  \|x\|_A=1\}.$$
For $T\in \mathcal{B(H)},$ an operator $R\in \mathcal{B(H)}$ is called an {\it $A$-adjoint operator} of $T$ if for every $x,y\in \mathcal{H},$ we have $\langle Tx,y\rangle_A=\langle x, Ry\rangle_A,$ i.e., $AR=T^*A.$
By Douglas Theorem \cite{Doug}, the existence of an $A$-adjoint operator is not guaranteed.  In fact, an operator $T\in \mathcal{B(H)}$ may admit none, one or many $A$-adjoints. The set of all operators which admits $A$-adjoint is denoted by $\mathcal{B}_A\mathcal{(H)}.$ Note that $\mathcal{B}_A\mathcal{(H)}$ is a subalgebra of $\mathcal{B(H)}$ which is neither closed nor dense in $\mathcal{B(H)}.$ Moreover, the following inclusions $\mathcal{B}_A\mathcal{(H)}\subseteq \mathcal{B}^A\mathcal{(H)}\subseteq\mathcal{B}\mathcal{(H)}$ hold with equality if $A$ is injective and has a closed range. 

For $A\in \mathcal{B(H)}$ and $\mathcal{R}(A)$  is closed, the {\it Moore-Penrose inverse} of $A$ \cite{gro} is the
operator $X\in \mathcal{B(H)}$ which satisfies the following four Penrose equations:
\begin{center}
(1) $AXA = A$,~ (2) $XAX = X$,~ (3) $(A X)^* = A X$,~ (4) $(X A)^*= X A.$
\end{center}
It
is unique, and  is denoted by $A^\dagger.$
If $T\in \mathcal{B}_A\mathcal{(H)},$ the reduced solution of the equation $AX=T^*A$ is a distinguished $A$-adjoint operator of $T,$ which is denoted by $T^{\#_A}$ (see \cite{Mos}). Note that $T^{\#_A}=A^\dagger T^* A$. If $T\in \mathcal{B}_A(\mathcal{H}),$ then $AT^{\#_A}=T^*A.$ An operator $T\in \mathcal{B(H)}$ is said to be {\it $A$-selfadjoint} if $AT$ is selfadjoint, i.e., $AT=T^*A.$ Observe that if $T$ is $A$-selfadjoint, then $T\in \mathcal{B}_A(\mathcal{H}).$  However,  in general, $T\neq T^{\#_A}.$ For $T\in \mathcal{B}_A(\mathcal{H}),$ $T=T^{\#_A}$  if and only if $T$ is $A$-selfadjoint and $\mathcal{R}(T)\subseteq \overline{\mathcal{R}(A)}.$ Note that if $T\in \mathcal{B}_A(\mathcal{H}),$ then $T^{\#_A}\in \mathcal{B}_A(\mathcal{H}),$  $(T^{\#_A})^{\#_A}=PTP,$ where $P$ is an orthogonal projection onto $\overline{\mathcal{R}(A)},$ and $\left((T^{\#_A})^{\#_A}\right)^{\#_A}=T^{\#_A}.$ Also $T^{\#_A}T$ and $TT^{\#_A}$ are $A$-selfadjoint and $A$-positive operators. So,
\begin{align}\label{ineq0}
    \|T^{\#_A}T\|_A=\|TT^{\#_A}\|_A=\|T\|_A^2=\|T^{\#_A}\|_A^2.
\end{align}
An operator $U\in \mathcal{B}_A(\mathcal{H})$ is said to be $A$-unitary if $\|Ux\|_A=\|U^{\#_A}x\|_A=\|x\|_A$ for all $x\in \mathcal{H}.$ For $T\in \mathcal{B}_A(\mathcal{H})$ and $U$ is $A$-unitary, $w_A(U^{\#_A}TU)=w_A(T).$

 Again, for $T,S\in \mathcal{B}_A(\mathcal{H}),$ $(TS)^{\#_A}=S^{\#_A}T^{\#_A},$ $\|TS\|_A\leq \|T\|_A\|S\|_A$ and $\|Tx\|_A\leq \|T\|_A\|x\|_A$ for all $x\in \mathcal{H}.$ For $T\in \mathcal{B}_A(\mathcal{H})$, we can write $Re_A(T)=\frac{T+T^{\#_A}}{2}$ and $Im_A(T)=\frac{T-T^{\#_A}}{2i}$. For further details, we refer the reader to \cite{ARIS,ARIS2}. 
 in 2012, Saddi \cite{Saddi} defined {\it $A-$numerical radius} of $T,$ denoted as $w_A(T),$  for $T\in \mathcal{B(H)}$ as follows
$$w_A(T)=\sup\{|\langle Tx,x\rangle_A|:x\in \mathcal{H}, \|x\|_A=1\}.
$$
In 2019, Zamani \cite{Zam} showed that if $T\in \mathcal{B}_A\mathcal{(H)}$, then 
\begin{align}\label{ineq00}
w_A(T)=\sup_{\theta\in \mathbb{R}}\left\|\frac{e^{i\theta}T+(e^{i\theta}T)^{\#_A}}{2}\right\|_A.
\end{align}
The author then extended the inequality \eqref{p3100} using $A$-numerical radius of $T$, and the same is illustrated next:
\begin{align}\label{ineq1}
    \frac{1}{2}\|T\|_A\leq w_A(T)\leq \|T\|_A.
\end{align}
Furthermore,  if $T$ is $A$-selfadjoint, then $w_A(T)=\|T\|_A$.
   In 2019,  Moslehian {\it et al.} \cite{MOS} further continued the study of $A$-numerical radius and established some inequalities for $A$-numerical radius.\\
   For a $2\times 2$ operator matrix $T,$  $B$-numerical radius of $T$ is defined as $$w_B(T)=\sup\{|\langle Tx,x\rangle_B|:x\in \mathcal{H}, \|x\|_B=1\},$$ where
    $B=\begin{bmatrix}
   A & 0\\
   0 & A
   \end{bmatrix}$.\\

   In 2019, Bhunia {\it et al.}  \cite{PINTU} studied $B$-numerical radius inequalities of $2\times 2$ operator matrices, where $B$ is a $2\times 2$ diagonal operator matrix whose diagonal entries are $A$. In this directions some authors has been studied many generalizations and refinements of $A$-numerical radius, for more details one can refer \cite{Pintu1, Feki, Pintu2}. This motivates us to further study on this topic.

The objective of this paper is
to present new $B$-numerical radius inequalities for  $2\times 2$ operator matrices. Further two refinements of the 1st inequality in \eqref{ineq1} is addressed in this article.
In this aspect, the article is  structured as follows. In Section 2,  we recall some upper and lower bounds for $B$-numerical radius inequalities for a $2\times 2 $ operator matrix.The next section contains our main results and is of two folds. First part establishes some upper and lower bounds for $2\times 2$ operator matrices while the second part deals with certain refinements of \eqref{ineq1}.  

\section{Preliminaries}
In 2020, Pintu {\it et al.} \cite{PINTU} proved the following lemma for $2\times 2$ operator matrices.
\begin{lemma}\label{lem0001}\textnormal{[Lemma 2.4 , \cite{PINTU}]} \\
Let $T_1, T_2\in \mathcal{B}_A(\mathcal{H}).$ Then the following results hold:
\begin{enumerate}
        \item [\textnormal{(i)}]
     $w_B\left(\begin{bmatrix}
    T_1 & 0\\
     0 & T_2
    \end{bmatrix}\right)= \max\{w_A(T_1), w_A(T_2)\}.$\\
    \item [\textnormal{(ii)}] If $A>0$, then $w_B\left(\begin{bmatrix}
    0 & T_1\\
     T_2 & 0
    \end{bmatrix}\right)=w_B\left(\begin{bmatrix}
    0 & T_2\\
     T_1 & 0
    \end{bmatrix}\right).$\\
    \item [\textnormal{(iii)}] If $A>0,$ then~for~any~$\theta\in\mathbb{R}, w_B\left(\begin{bmatrix}
    0 & T_1\\
     e^{i\theta}T_2 & 0
    \end{bmatrix}\right)=w_B\left(\begin{bmatrix}
    0 & T_1\\
     T_2 & 0
    \end{bmatrix}\right).$\\
    \item [\textnormal{(iv)}] If $A>0$,~ then~ $w_B\left(\begin{bmatrix}
    T_1 & T_2\\
     T_2 & T_1
    \end{bmatrix}\right)=\max\{w_A(T_1+T_2),w_A(T_1-T_2)\}.$\\
     In particular, $w_B\left(\begin{bmatrix}
    0 & T_1\\
     T_1 & 0
    \end{bmatrix}\right)=w_A(T_1).$
\end{enumerate}
\end{lemma}
 
In 2019, the authors of \cite{Pintu1} established an upper and lower bound for a $2\times 2$ operator matrix. \begin{lemma}\label{lem0002}\textnormal{[Theorem 4.3, \cite{Pintu1}]} \\
Let $T_1, T_2\in \mathcal{B}_A(\mathcal{H})$ where $A>0.$ If $T=\begin{bmatrix}
    0 & T_1\\
     T_2 & 0
    \end{bmatrix}$ and $B=\begin{bmatrix}
    0 & A\\
     A & 0
    \end{bmatrix}$ then
    \begin{align*}
    \frac{1}{2} \max\{w_A(T_1+T_2), w_A(T_1-T_2)\}&\leq w_B(T)\\
    &\leq \frac{1}{2}\{w_A(T_1+T_2)+ w_A(T_1-T_2)\}.
    \end{align*}
\end{lemma}
In 2020, Feki \cite{Feki1} proved the following result.
\begin{lemma}\label{lem00003}\textnormal{[Lemma 2.1, \cite{Feki1}]} \\
Let $T=(T_{ij})_{n\times n}$ such that $T_{ij}\in \mathcal{B}_A(\mathcal{H})$ for all $i,j.$ Then 
$$\|T\|_A\leq \|\widehat{T}\|,$$
where $\widehat{T}=(\|T_{ij}\|_A)_{n\times n}.$
\end{lemma}

 \section{Main Results}
	This section is two fold.  First, we present some generalizations of $A$-numerical radius inequalities. Further we prove some upper and lower bounds for $B$-numerical radius  of operator matrices.  Second, we provide different refinements of $A$-numerical radius inequalities.
 	
 \subsection{Upper and lower bounds for $B$-numerical radius of $2\times 2$ operator matrix.}
 
 In this subsection, we establish different upper and lower bounds for $B$-numerical radius of a $2\times 2$ block operator matrix. We start with the following lemma. 
 \begin{lemma}\label{l001}
 Let $T_1, T_2, T_3, T_4\in \mathcal{B}_A(\mathcal{H}).$ Then
  \begin{enumerate}
        \item [\textnormal{(i)}] $w_B\left(\begin{bmatrix}
     T_1 & 0\\
     0 & T_4
     \end{bmatrix}\right)\leq w_B\left(\begin{bmatrix}
     T_1 & T_2\\
     T_3 & T_4
     \end{bmatrix}\right).$
      \item [\textnormal{(ii)}] $w_B\left(\begin{bmatrix}
     0 & T_2\\
     T_3 & 0
     \end{bmatrix}\right)\leq w_B\left(\begin{bmatrix}
     T_1 & T_2\\
     T_3 & T_4
     \end{bmatrix}\right).$
        \end{enumerate}
 \end{lemma}
 
 \begin{proof}
 Let $ T= \begin{bmatrix}
     T_1 & T_2\\
     T_3 & T_4
     \end{bmatrix}$ and the $B$-unitary operator $U=\begin{bmatrix}
     I & 0\\
     0 & -I
     \end{bmatrix}.$\\
     Here, $
     \begin{bmatrix}
     T_1 & 0\\
     0 & T_4
     \end{bmatrix}=\frac{1}{2}(T+U^{\#_B}TU).$ So, we have \\
     (i)
     \begin{align*}
         w_B\left( \begin{bmatrix}
     T_1 & 0\\
     0 & T_4
     \end{bmatrix}\right)&= \frac{1}{2}w_B(T+U^{\#_B}TU)\\
     &\leq \frac{1}{2}[w_B(T)+w_B(U^{\#_B}TU)]\\
     &= \frac{1}{2}[w_B(T)+w_B(T)]\\
     &=w_B(T)=w_B\left(\begin{bmatrix}
     T_1 & T_2\\
     T_3 & T_4
     \end{bmatrix}\right).
     \end{align*} 
    (ii) 
      \begin{align*}
         w_B\left( \begin{bmatrix}
     0 & T_2\\
     T_3 & 0
     \end{bmatrix}\right)&= \frac{1}{2}w_B(T-U^{\#_B}TU)\\
     &\leq \frac{1}{2}[w_B(T)+w_B(U^{\#_B}TU)]\\
     &= \frac{1}{2}[w_B(T)+w_B(T)]\\
     &=w_B(T)=w_B\left(\begin{bmatrix}
     T_1 & T_2\\
     T_3 & T_4
     \end{bmatrix}\right).
     \end{align*} 
 \end{proof}

The following inequality generalizes  \eqref{ineq1}.

\begin{thm}
Let $T_1, T_2\in\mathcal{B}_A(\mathcal{H}),$ where $A>0.$
If $B=\begin{bmatrix}
A & 0\\
0 & A
\end{bmatrix}
,$ then 
\begin{equation}\label{eq01}
    \max\{w_A(T_1), w_A(T_2)\}\leq w_B\left(\begin{bmatrix}
T_1 & T_2\\
-T_2 & -T_1
\end{bmatrix}\right)\leq w_A(T_1)+w_A(T_2).
\end{equation}
\end{thm}
\begin{proof}
By using Lemma \ref{lem0001} and Lemma \ref{l001}
$$w_A(T_1)= w_B\left(\begin{bmatrix}
T_1 & 0\\
0 & -T_1
\end{bmatrix}\right)\leq w_B\left(\begin{bmatrix}
T_1 & T_2\\
-T_2 & -T_1
\end{bmatrix}\right)$$
and
$$w_A(T_2)= w_B\left(\begin{bmatrix}
0 & T_2\\
-T_2 & 0
\end{bmatrix}\right)\leq w_B\left(\begin{bmatrix}
T_1 & T_2\\
-T_2 & -T_1
\end{bmatrix}\right).$$
Therefore, $$\max\{w_A(T_1), w_A(T_2)\}\leq w_B\left(\begin{bmatrix}
T_1 & T_2\\
-T_2 & -T_1
\end{bmatrix}\right).$$
On the other hand, by using Lemma \ref{lem0001}, we have 
$$w_B\left(\begin{bmatrix}
T_1 & T_2\\
-T_2 & -T_1
\end{bmatrix}\right)\leq w_B\left(\begin{bmatrix}
T_1 & 0\\
0 & -T_1
\end{bmatrix}\right)+w_B\left(\begin{bmatrix}
0 & T_2\\
-T_2 & 0
\end{bmatrix}\right)=w_A(T_1)+w_A(T_2).$$
\end{proof}
A particular case of the inequality \eqref{eq01} is the following.

\begin{remark}
If we choose $T_2=T_1$ in inequality \eqref{eq01}, then 
$$w_A(T_1)\leq w_B\left(\begin{bmatrix}
T_1 & T_1\\
-T_1 & -T_1
\end{bmatrix}\right)\leq 2w_A(T_1).$$
\end{remark}
We need the following lemma to prove Theorem \ref{t002}.

\begin{lemma}\label{l002}
Let $T_1, T_2, T_3, T_4\in \mathcal{B}_A(\mathcal{H}),$ where $A>0.$ If $B=\begin{bmatrix}
A & 0\\
0 & A
\end{bmatrix},$ then
$$ w_B\left(\begin{bmatrix}
T_2 & -T_1\\
T_1 & T_2
\end{bmatrix}\right)= \max\{w_A(T_1+iT_2), w_A(T_1-iT_2)\}.$$
\end{lemma}
\begin{proof}
Let $T=\begin{bmatrix}
iT_2 & -T_1\\
T_1 & iT_2
\end{bmatrix}$ and the $B$-unitary operator  $U=\frac{1}{\sqrt{2}}\begin{bmatrix}
I & iI\\
iI & I
\end{bmatrix}.$ Then $U^{\#_B}TU=\begin{bmatrix}
-i(T_1-T_2) & 0\\
0 & i(T_1+T_2)
\end{bmatrix}.$ Using the fact that $w_B(T)=w_B(U^{\#_B}TU),$
we get
\begin{align*}
    w_B(T)=w_B(U^{\#_B}TU)&=
w_B\left(\begin{bmatrix}
-i(T_1-T_2) & 0\\
0 & i(T_1+T_2)
\end{bmatrix}\right)\\
&=\max\{w_A(-i(T_1-T_2)),w_A(i(T_1+T_2))\}\\ &=\max\{w_A(T_1-T_2),w_A(T_1+T_2)\}.
\end{align*}
Replacing $T_2$ by $-iT_2$ in the identity, we have
$$
w_B\left(\begin{bmatrix}
T_2 & -T_1\\
T_1 & T_2
\end{bmatrix}\right)
=\max\{w_A(T_1+iT_2),w_A(T_1-iT_2)\}.$$
\end{proof}
Theorem \ref{t002} provides an upper bound for a block operator matrix of the form $\begin{bmatrix}
T_1 & T_2\\
T_3 & T_4
\end{bmatrix}.$
\begin{thm}\label{t002}
Let $T_1, T_2, T_3, T_4\in \mathcal{B}_A(\mathcal{H}),$ where $A>0.$ If $T=\begin{bmatrix}
T_1 & T_2\\
T_3 & T_4
\end{bmatrix}$ and $B=\begin{bmatrix}
A & 0\\
0 & A
\end{bmatrix}.$ Then
\begin{align*}
 w_B(T)\leq \max\bigg\{\frac{1}{2}w_A(T_1+T_4+i(T_2-T_3)),\frac{1}{2}w_A(T_1+&T_4-i(T_2-T_3))\bigg\}  \\ &+\frac{1}{2}(w_A(T_4-T_1)+w_A(T_2+T_3)).  
\end{align*}

\end{thm}

\begin{proof}
Let $U=\frac{1}{\sqrt{2}}\begin{bmatrix}
I & -I\\
I & I
\end{bmatrix}$ be $B$-unitary. Using the identity $w_B(T)=w_B(U^{\#_B}TU)$, we have
\begin{align*}
    w_B\left(\begin{bmatrix}
T_1 & T_2\\
T_3 & T_4
\end{bmatrix}\right)=& w_B\left(U^{\#_B}\begin{bmatrix}
T_1 & T_2\\
T_3 & T_4
\end{bmatrix}U\right)\\
&=\frac{1}{2}w_B\left(\begin{bmatrix}
T_1+T_2+T_3+T_4 & -T_1+T_2-T_3+T_4\\
-T_1-T_2+T_3+T_4 & T_1-T_2-T_3+T_4
\end{bmatrix}\right)\\
&=\frac{1}{2}w_B\left(\begin{bmatrix}
T_1+T_4 & T_2-T_3\\
T_3-T_2 & T_1+T_4
\end{bmatrix}+\begin{bmatrix}
T_2+T_3 & T_4-T_1\\
T_4-T_1 & -T_3-T_2
\end{bmatrix}\right)\\
&\leq \frac{1}{2}\left\{w_B\left(\begin{bmatrix}
T_1+T_4 & T_2-T_3\\
T_3-T_2 & T_1+T_4
\end{bmatrix}\right)+w_B\left(\begin{bmatrix}
T_2+T_3 & T_4-T_1\\
T_4-T_1 & -T_3-T_2
\end{bmatrix}\right)\right\}\\
&\leq \frac{1}{2}\{\max(w_A(T_3-T_2+i(T_1+T_4))),w_A(T_3-T_2-i(T_1+T_4))\\ &\hspace{2.5cm}+w_A(T_4-T_1)+w_A(T_2+T_3)\}\mbox{~~by Lemma \ref{l002} and Lemma \ref{lem0001}.}
\end{align*}
\end{proof}

The following result demonstrates an upper bound for $B$-numerical radius of a $2\times 2$ operator matrix.
\begin{thm}
Let $T_1,T_2,T_3,T_4\in \mathcal{B}_A(\mathcal{H}),$ where $A>0.$ If $B=\begin{bmatrix}
A & 0\\
0 & A
\end{bmatrix}.$ Then
$$w_B\left(\begin{bmatrix}
T_1 & T_2\\
T_3 & T_4
\end{bmatrix}\right)\leq \max\{w_A(T_1), w_A(T_4)\}+\frac{w_A(T_2+T_3)+w_A(T_2-T_3)}{2}.$$
\end{thm}
\begin{proof}
Using similar argument as used in the previous theorem, we have
\begin{align*}
    w_B\left(\begin{bmatrix}
T_1 & T_2\\
T_3 & T_4
\end{bmatrix}\right)&=\frac{1}{2}w_B\left(\begin{bmatrix}
T_1+T_4 & T_4-T_1\\
T_4-T_1 & T_1+T_4
\end{bmatrix}+\begin{bmatrix}
T_2+T_3 & T_2-T_3\\
T_3-T_2 & -T_3-T_2
\end{bmatrix}\right)\\
&\leq \frac{1}{2}\left[w_B\left(\begin{bmatrix}
T_1+T_4 & T_4-T_1\\
T_4-T_1 & T_1+T_4
\end{bmatrix}\right)+w_B\left(\begin{bmatrix}
T_2+T_3 & T_2-T_3\\
T_3-T_2 & -T_3-T_2
\end{bmatrix}\right)\right]\\
&\leq \frac{1}{2}\max\{w_A(T_1+T_4+T_4-T_1),w_A(T_1+T_4-T_4+T_1)\}\\
& \hspace{4cm}+\frac{1}{2}\{w_A(T_2+T_3)+w_A(T_2-T_3)\}\mbox{ by Lemma \ref{lem0001}(iv)}\\
&=\max\{w_A(T_1),w_A(T_4)\}+\frac{w_A(T_2+T_3)+w_A(T_2-T_3)}{2}.
\end{align*}
\end{proof}

The following result is an estimate of an lower bound for $B$-numerical radius  of a $2\times 2$ operator matrix.
\begin{thm}
Let $T_1, T_2, T_3,T_4\in \mathcal{B}_A(\mathcal{H}),$ where $A>0.$ If 
$B=\begin{bmatrix}
A & 0\\
0 & A
\end{bmatrix},$  
 then
\begin{align*}
  w_B\left(\begin{bmatrix}
T_1 & T_2\\
T_3 & T_4
\end{bmatrix}\right)\geq \max\left\{w_A(T_1), w_A(T_4)),\frac{w_A(T_2+T_3)}{2}, \frac{w_A(T_2-T_3)}{2}\right\}.
\end{align*}
\end{thm}
\begin{proof}
It follows from Lemma \ref{l001} that 
\begin{align*}
    w_B\left(\begin{bmatrix}
T_1 & T_2\\
T_3 & T_4
\end{bmatrix}\right) & \geq \max \left\{w_B\left(\begin{bmatrix}
T_1 & 0\\
0 & T_4
\end{bmatrix}\right), w_B\left(\begin{bmatrix}
0 & T_2\\
T_3 & 0
\end{bmatrix}\right)\right\}\\
&=\max\left\{\max\{w_A(T_1), w_A(T_4)\}, w_B\left(\begin{bmatrix}
0 & T_2\\
T_3 & 0
\end{bmatrix}\right)\right\}\\
& \geq \max\left\{\max\{w_A(T_1), w_A(T_4)\}, \frac{\max(w_A(T_2+T_3),w_A(T_2-T_3))}{2}\right\} \mbox{by Lemma \ref{lem0002}}\\
&=\max\left\{w_A(T_1), w_A(T_4), \frac{w_A(T_2+T_3)}{2},\frac{w_A(T_2-T_3)}{2}\right\}.\\
& \hspace{7cm}
\end{align*}
\end{proof}
To prove the next lemma, we need the following identities,   
\begin{equation}\label{eq001}
    \frac{a+b}{2}=\max(a,b)-\frac{|a-b|}{2}
\end{equation}
and
\begin{equation}\label{eq002}
    \frac{a+b}{2}=\min(a,b)+\frac{|a-b|}{2},
\end{equation}
for any two real numbers $a$ and $b$.
\begin{lemma}
    Let $T_1,T_2\in \mathcal{B}_A(\mathcal{H}).$ Then
    \begin{align*}
\max(\|T_1&+T_2\|_A^2, \|T_1-T_2\|_A^2)\\ &\leq \min(\|T_1^{\#_A}T_1+T_2^{\#_A}T_2\|_A+\|T_1^{\#_A}T_2+ T_2^{\#_A}T_1\|_A, \|T_1T_1^{\#_A}+T_2T_2^{\#_A}\|_A+\| T_1T_2^{\#_A}+T_2T_1^{\#_A}\|_A)
\end{align*}
and
\begin{align*}
\min(\|T_1&+T_2\|_A^2, \|T_1-T_2\|_A^2)\\ &\geq \max(\|T_1^{\#_A}T_1+T_2^{\#_A}T_2\|_A-\|T_1^{\#_A}T_2+ T_2^{\#_A}T_1)\|_A, \|T_1T_1^{\#_A}+T_2T_2^{\#_A}\|_A-\| T_1T_2^{\#_A}+T_2T_1^{\#_A}\|_A).
\end{align*}
\end{lemma}

\begin{lemma}\label{lem0004}
Let $T_1, T_2\in \mathcal{B}_A(\mathcal{H}).$ Then 
\begin{align*}
     \max(\|T_1+T_2\|_A^2, \|T_1-T_2\|_A^2)\geq \max(\|T_1^2+T_2^2\|_A, &\|T_1^{\#_A}T_1+T_2^{\#_A}T_2\|_A,\|T_1T_1^{\#_A}+T_2T_2^{\#_A}\|_A)\\
    &\hspace{3cm}+\frac{|\|T_1+T_2\|_A^2-\|T_1-T_2\|_A^2|}{2}.
\end{align*}
and
\begin{align*}
    \min(\|T_1+T_2\|_A^2, \|T_1-T_2\|_A^2)\geq \max(\|T_1^2+T_2^2\|_A, &\|T_1^{\#_A}T_1+T_2^{\#_A}T_2\|_A,\|T_1T_1^{\#_A}+T_2T_2^{\#_A}\|_A)\\
    &\hspace{3cm}-\frac{|\|T_1+T_2\|_A^2-\|T_1-T_2\|_A^2|}{2}.
\end{align*}
\end{lemma}

Following theorem demonstrates an upper bound for $B$-numerical radius of $2\times 2$ operator matrix using \eqref{ineq1} and Lemma \ref{lem0004}.
\begin{thm}\label{thm20005}
Let $T_1,T_2,T_3,T_4\in\mathcal{B}_A(\mathcal{H}).$ Then 
$$w_B\left(\begin{bmatrix}
    T_1 & T_2\\
    T_3 & T_4
    \end{bmatrix}\right)\leq \min(\alpha,\beta),$$
    where, $$\alpha=\left(\frac{\|T_1+T_2\|_A^2+\|T_1-T_2\|_A^2}{2}\right)^{\frac{1}{2}}+\left(\frac{\|T_4+T_3\|_A^2+\|T_4-T_3\|_A^2}{2}\right)^{\frac{1}{2}}$$
    and 
    $$\beta=\left(\frac{\|T_1+T_3\|_A^2+\|T_1^-T_3\|_A^2}{2}\right)^{\frac{1}{2}}+\left(\frac{\|T_2+T_4^{\#_A}\|_A^2+\|T_2-T_4^{\#_A}\|_A^2}{2}\right)^{\frac{1}{2}}.$$
\end{thm}
\begin{proof}
We know that
\begin{align*}
    w_B\left(\begin{bmatrix}
    T_1& T_2\\
    0& 0
    \end{bmatrix}\right)&\leq \left\|\begin{bmatrix}
    T_1& T_2\\
    0 & 0
    \end{bmatrix}\right\|_B \\
    &=\left\|\begin{bmatrix}
    T_1& T_2\\
    0& 0
    \end{bmatrix}\begin{bmatrix}
    T_1& T_2\\
    0& 0
    \end{bmatrix}^{\#_B}\right\|_B^{\frac{1}{2}}\\
    &=\left\|\begin{bmatrix}
    T_1& T_2\\
    0& 0
    \end{bmatrix}\begin{bmatrix}
    T_1^{\#_A} & 0\\
    T_2^{\#_A}& 0
    \end{bmatrix}\right\|_B^{\frac{1}{2}}\\
    &=\left\|\begin{bmatrix}
    T_1T_1^{\#_A}+T_2T_2^{\#_A} & 0\\
    0& 0
    \end{bmatrix}\right\|_B^{\frac{1}{2}}\\
    &=\|T_1T_1^{\#_A}+T_2T_2^{\#_A}\|_A^{\frac{1}{2}}.
\end{align*}
By using Lemma \ref{lem0004}, we get 
\begin{align*}
    w_B\left(\begin{bmatrix}
    T_1 & T_2\\
    0& 0
    \end{bmatrix}\right)&\leq \left(\max(\|T_1+T_2\|_A^2,\|T_1-T_2\|_A^2)-\frac{|\|T_1+T_2\|_A^2-\|T_1-T_2\|_A^2|}{2}\right)^{\frac{1}{2}}\\
    &=\left(\frac{\|T_1+T_2\|_A^2+\|T_1-T_2\|_A^2}{2}\right)^{\frac{1}{2}}.
\end{align*}
Let us take $U=\begin{bmatrix}
    0 & I\\
    I & 0
    \end{bmatrix},$ where $U$ is $B$-unitary.
Now, we have
\begin{align*}
    w_B\left(\begin{bmatrix}
    T_1 & T_2\\
    T_3 & T_4
    \end{bmatrix}\right)&\leq w_B\left(\begin{bmatrix}
    T_1 & T_2\\
    0 & 0
    \end{bmatrix}\right)+w_B\left(\begin{bmatrix}
    0 & 0\\
    T_3 & T_4
    \end{bmatrix}\right)\\
    &= w_B\left(\begin{bmatrix}
    T_1 & T_2\\
    0 & 0
    \end{bmatrix}\right)+w_B\left(U^{\#_B}\begin{bmatrix}
    T_4 & T_3\\
    0 & 0
    \end{bmatrix}U\right) \\
    &=w_B\left(\begin{bmatrix}
    T_1 & T_2\\
    0 & 0
    \end{bmatrix}\right)+w_B\left(\begin{bmatrix}
    T_4 & T_3\\
    0 & 0
    \end{bmatrix}\right)\\
    &\leq \left(\frac{\|T_1+T_2\|_A^2+\|T_1-T_2\|_A^2}{2}\right)^{\frac{1}{2}}+\left(\frac{\|T_4+T_3\|_A^2+\|T_4-T_3\|_A^2}{2}\right)^{\frac{1}{2}}=\alpha
\end{align*}
Applying the previous calculation to $\begin{bmatrix}
    T_1^{\#_A} & T_3^{\#_A}\\
    T_2^{\#_A} & T_4^{\#_A}
    \end{bmatrix}$ in the place of $\begin{bmatrix}
    T_1 & T_2\\
    T_3 & T_4
    \end{bmatrix}$, we obtain
    \begin{align*}
         w_B\left(\begin{bmatrix}
    T_1 & T_2\\
    T_3 & T_4
    \end{bmatrix}\right)&= w_B\left(\begin{bmatrix}
    T_1^{\#_A} & T_3^{\#_A}\\
    T_2^{\#_A} & T_4^{\#_A}
    \end{bmatrix}\right)\\
    &\leq \left(\frac{\|T_1^{\#_A}+T_3^{\#_A}\|_A^2+\|T_1^{\#_A}-T_3^{\#_A}\|_A^2}{2}\right)^{\frac{1}{2}}+\left(\frac{\|T_2^{\#_A}+T_4^{\#_A}\|_A^2+\|T_2^{\#_A}-T_4^{\#_A}\|_A^2}{2}\right)^{\frac{1}{2}}\\
    &=\left(\frac{\|T_1+T_3\|_A^2+\|T_1-T_3\|_A^2}{2}\right)^{\frac{1}{2}}+\left(\frac{\|T_2+T_4\|_A^2+\|T_2-T_4\|_A^2}{2}\right)^{\frac{1}{2}}=\beta.
    \end{align*}
    Hence, we get the desired  result.
\end{proof}
Next result shows a lower bound for $B$-numerical radius of a $2\times 2$ operator matrix in which 2nd row is zero.
\begin{thm}
Let $T_1,T_2\in\mathcal{B}_A(\mathcal{H}).$ Then
$$w_B\left(\begin{bmatrix}
T_1 & T_2\\
0 & 0
\end{bmatrix}\right)\geq \frac{1}{2}\max(w_A(T_1\pm T_2),w_A(T_1\pm iT_2)).$$
\end{thm}
\begin{proof}
Let $U=\begin{bmatrix}
0 & I\\I & 0
\end{bmatrix}$ a $B$-unitary operator. 
Then
\begin{align*}
    \max(w_A(T_1+T_2),w_A(T_1-T_2))&=w_B\left(\begin{bmatrix}
    T_1 & T_2\\T_2 & T_1
    \end{bmatrix}\right) \mbox{~~by Lemma \ref{lem0001}}\\
    &=w_B\left(\begin{bmatrix}
    T_1 & T_2\\0 & 0
    \end{bmatrix}+U^{\#_B}\begin{bmatrix}
    T_1 & T_2\\0 & 0
    \end{bmatrix}U\right)\\
    &\leq w_B\left(\begin{bmatrix}
    T_1 & T_2\\0 & 0
    \end{bmatrix}\right)+w_B\left(U^{\#_B}\begin{bmatrix}
    T_1 & T_2\\0 & 0
    \end{bmatrix}U\right)\\
    &=2w_B\left(\begin{bmatrix}
    T_1 & T_2\\0 & 0
    \end{bmatrix}\right).
\end{align*}
Setting $V=\begin{bmatrix}
    I & 0\\0 & -I
    \end{bmatrix},$ it is not difficult to see that $V$ is $B$-unitary.
    Now,
    \begin{align*}
        \max(w_A(T_1+iT_2), w_A(T_1-iT_2))=&w_B\left(\begin{bmatrix}
    T_1 & -T_2\\T_2 & T_1
    \end{bmatrix}\right) \mbox{~~by Lemma \ref{l002}}\\
    &=w_B\left(V^{\#_A}\begin{bmatrix}
    T_1 & T_2\\0 & 0
    \end{bmatrix}V+U^{\#_A}\begin{bmatrix}
    T_1 & T_2\\0 & 0
    \end{bmatrix}U\right)\\
    &\leq w_B\left(V^{\#_B}\begin{bmatrix}
    T_1 & T_2\\0 & 0
    \end{bmatrix}V\right )+w_B\left(U^{\#_B}\begin{bmatrix}
    T_1 & T_2\\0 & 0
    \end{bmatrix}U\right)\\
    &=2w_B\left(\begin{bmatrix}
    T_1 & T_2\\0 & 0
    \end{bmatrix}\right).
    \end{align*}
   Hence, we get
   $$w_B\left(\begin{bmatrix}
    T_1 & T_2\\0 & 0
    \end{bmatrix}\right)\geq \frac{1}{2}\max(w_A(T_1\pm T_2), w_A(T_1\pm iT_2)).$$
\end{proof}
Further upper bound for the $B$-numerical radius of $\begin{bmatrix}
T_1 & T_2\\
T_3 & T_4
\end{bmatrix}$ is proved next using the Lemma \ref{lem00003}.
\begin{thm}\label{thm20006}
Let $T_1,T_2\in \mathcal{B}_A(\mathcal{H}).$ 

   $$w_B\left(\begin{bmatrix}
T_1 & T_2\\
T_3 & T_4
\end{bmatrix}\right) \leq \min\{\alpha,\beta\},$$
where, 
\begin{align*}
    \alpha=&\frac{1}{\sqrt{2}}\sqrt{\|T_1\|_A^2+\|T_2\|_A^2+\sqrt{(\|T_1\|_A^2-\|T_2\|_A^2)^2+4\|T_1^{\#_A}T_2\|_A^2}}\\&+\frac{1}{\sqrt{2}}\sqrt{\|T_3\|_A^2+\|T_4\|_A^2+\sqrt{(\|T_3\|_A^2-\|T_4\|_A^2)^2+4\|T_4T_3^{\#_A}\|_A^2}}
\end{align*}
and
\begin{align*}
    \beta=&\frac{1}{\sqrt{2}}\sqrt{\|T_1\|_A^2+\|T_3\|_A^2+\sqrt{(\|T_1\|_A^2-\|T_3\|_A^2)^2+4\|T_1^{\#_A}T_3\|_A^2}}\\&+\frac{1}{\sqrt{2}}\sqrt{\|T_2\|_A^2+\|T_4\|_A^2+\sqrt{(\|T_2\|_A^2-\|T_4\|_A^2)^2+4\|T_4^{\#_A}T_2\|_A^2}}.
\end{align*}
\end{thm}

We now give an special case of Theorem \ref{thm20006} in the following corollary.
\begin{cor}
Let $T_1,T_2,T_3,T_4\in \mathcal{B}_A(\mathcal{H}).$\\
(a) If $T_1^{\#_A}T_2=0= T_4T_3^{\#_A},$ then 
$$w_B\left(\begin{bmatrix}
T_1 & T_2\\
T_3 & T_4
\end{bmatrix}\right) \leq \max(\|T_1\|_A, \|T_2\|_A)+\max(\|T_3\|_A,\|T_4\|_A).$$
(b) If $T_1^{\#_A}T_3=0= T_4^{\#_A}T_2,$ then 
$$w_B\left(\begin{bmatrix}
T_1 & T_2\\
T_3 & T_4
\end{bmatrix}\right) \leq \max(\|T_1\|_A, \|T_3\|_A)+\max(\|T_2\|_A,\|T_4\|_A).$$
\end{cor}
\begin{remark}
Note that  equality holds in \textnormal{Theorem \ref{thm20005}} and \textnormal{Theorem \ref{thm20006}} by setting $T_1=I,T_2=T_3=T_4=0$ . So $A$-numerical radius inequalities for $2\times 2$ operator matrices in  \textnormal{Theorem \ref{thm20005}} and \textnormal{Theorem \ref{thm20006}} are sharp. 
\end{remark}

\subsection{Refinements of $A$-numerical radius inequality for an operator}
In this subsection, we present two  refinements of \eqref{ineq1}. To do this, we need the  identity \eqref{eq001}.

The first refinement of inequality \eqref{ineq1} is proved next.
\begin{thm}
Let $T_1,T_2\in \mathcal{B}_A(\mathcal{H}).$ Then
\begin{align*}
    w_B\left(\begin{bmatrix}
    0 & T_1\\
    T_2 & 0
    \end{bmatrix}\right) 
    &\leq w_A(T_1)+w_A(T_2)-\frac{1}{2}|w_A(T_1+T_2)-w_A(T_1-T_2)|.
\end{align*}
In particular
$$\frac{\|T_1\|_A}{2}+\frac{\|Re T_1^{\#_A}\|_A-\|Im T_1^{\#_A}\|_A}{2}\leq w_A(T_1).$$
\end{thm}

\begin{proof}
By  Lemma \ref{lem0002} and Equality \eqref{eq001}, we have
\begin{align*}
    w_B\left(\begin{bmatrix}
    0 & T_1\\
    T_2 & 0
    \end{bmatrix}\right) &\leq \frac{1}{2} \{w_A(T_1+T_2)+w_A(T_1-T_2)\}\\
    &=\max\{w_A(T_1+T_2), w_A(T_1-T_2)\}-\frac{1}{2}|w_A(T_1+T_2)-w_A(T_1-T_2)|\\
    &\leq w_A(T_1)+w_A(T_2)-\frac{1}{2}|w_A(T_1+T_2)-w_A(T_1-T_2)|.
\end{align*}
Replacing  $T_1$ by $T_1^{\#_A}$ and $T_2$ by $(T_1^{\#_A})^{\#_A}$, we get
\begin{align*}
    & w_B\left(\begin{bmatrix}
    0 & T_1^{\#_A}\\
    (T_1^{\#_A})^{\#_A} & 0
    \end{bmatrix}\right)\\ 
    &\leq w_A(T_1^{\#_A})+w_A((T_1^{\#_A})^{\#_A})-\frac{1}{2}|w_A(T_1^{\#_A}+(T_1^{\#_A})^{\#_A})-w_A(T_1^{\#_A}-(T_1^{\#_A})^{\#_A})|\\
    & =2w_A(T_1^{\#_A})-|\|Re T_1^{\#_A}\|_A-\|Im T_1^{\#_A}\|_A|.
\end{align*}
 So, $$\frac{\|T_1\|_A}{2}+\frac{\|Re T_1^{\#_A}\|_A-\|Im T_1^{\#_A}\|_A}{2}\leq w_A(T_1^{\#_A})=w_A(T_1).$$
\end{proof}

Another refinement of the Inequality \eqref{ineq1} is presented next.
\begin{thm}
Let $T_1, T_2\in \mathcal{B}_A(\mathcal{H}).$ Then
\begin{align*}
w_B\left(\begin{bmatrix}
    0 & T_1\\
    T_2 & 0
    \end{bmatrix}\right)+\frac{\|T_1\|_A+\|T_2\|_A}{2}+\frac{1}{2}\left|w_A(T_1+T_2)-\frac{\|T_1\|_A+\|T_2\|_A}{2}\right|+&\frac{1}{2}\left|w_A(T_1-T_2)-\frac{\|T_1\|_A+\|T_2\|_A}{2}\right|\\
    &\leq2(w_A(T_1)+w_A(T_2)).
    \end{align*}
    In particular, 
$$\frac{\|T_1\|_A}{2}+\frac{1}{4}\left|\|Re( T_1^{\#_A})\|_A-\frac{\|T_1\|_A}{2}\right|+\frac{1}{4}\left|\|Im( T_1^{\#_A})\|_A-\frac{\|T_1\|_A}{2}\right|\leq w_A(T_1).$$
\end{thm}

\vspace{.6cm}
\noindent
{\small {\bf Acknowledgments.}\\
We  thank the {\bf Government of India} for introducing the {\it work from home initiative} during the COVID-19 crisis.
}
\section{References}
	\bibliographystyle{amsplain}

\end{document}